%% file: InvertibleCarnotGroupsUpdated.tex
\documentclass{amsart}
\usepackage{amssymb,amscd}
\usepackage{pstricks}
\usepackage{graphicx}
\usepackage{amsthm}

\numberwithin{equation}{section}

\newtheorem{theorem}{Theorem}[section]

\newtheorem{lemma}[theorem]{Lemma}

\newtheorem{fact}[theorem]{Fact}

\theoremstyle{remark}
\newtheorem*{acknowledgments}{Acknowledgments}

\DeclareMathOperator{\Lie}{Lie}

\DeclareMathOperator{\End}{End}

\DeclareMathOperator{\Span}{Span}

\input{macros.tex}

\title{Invertible Carnot Groups}
\date{\today}   

\author{David M. Freeman}
\address{%
\footnotesize
\begin{tabular}{lr}
\textsf{University of Cincinnati Blue Ash College}\\
\textsf{9555 Plainfield Road, Cincinnati, Ohio 45236}
\end{tabular}
\normalsize}
\email{david.freeman@uc.edu}

\begin{document} 


\begin{abstract}
We characterize Carnot groups admitting a $1$-quasiconformal metric inversion as the Lie groups of Heisenberg type whose Lie algebras satisfy the $J^2$-condition, thus characterizing a special case of inversion invariant bi-Lipschitz homogeneity. A more general characterization of inversion invariant bi-Lipschitz homogeneity for certain non-fractal metric spaces is also provided.
\end{abstract}

\maketitle

\section{Introduction}\label{S:intro}      

In \cite[Theorem 5.1]{CDKR-Heisenberg}, the authors characterize the nilpotent components of Iwasawa decompositions of real rank one simple Lie groups by the fact that they are of Heisenberg type and admit certain conformal inversions. The conformal inversions of \cite[Theorem 5.1]{CDKR-Heisenberg} generalize the usual M\"obius inversions of Euclidean space. Both of these notions of inversion are generalized in turn by the concept of \textit{metric inversion} as studied in \cite{BHX-inversions}. In the present paper, we study Carnot groups that admit such metric inversions.

The Carnot groups of \cite{CDKR-Heisenberg} are equipped with a left-invariant gauge distance that is bi-Lipschitz equivalent to a left-invariant sub-Riemannian (or more generally, sub-Finsler) distance. The study of such left-invariant distances can be generalized by the study of metric spaces that admit a transitive family of uniformly bi-Lipschitz self-homeomorphisms. These spaces are said to be \textit{uniformly bi-Lipschitz homogeneous}.  

The concepts of metric inversion and bi-Lipschitz homogeneity are combined in the notion of \textit{inversion invariant bi-Lipschitz homogeneity} as studied in \cite{Freeman-iiblh} and \cite{Freeman-params}. In particular, \cite[Theorem 2.4]{Freeman-params} demonstrates that certain inversion invariant bi-Lipschitz homogeneous geodesic spaces are bi-Lipschitz equivalent to Carnot groups (cf. \cite{Ledonne-geodesic}, \cite{Ledonne-tangents}). However, \cite[Theorem 2.4]{Freeman-params} does not provide a characterization of Carnot groups via inversion invariant bi-Lipschitz homogeneity. Indeed, in light of \cite[Theorem 5.1]{CDKR-Heisenberg}, the invertibility of a Carnot group appears to be a very restrictive condition. The following theorem identifies the Carnot groups admitting a metric inversion under the additional assumption that the inversion is $1$-quasiconformal (see \rf{S:defs} and \rf{S:heisenberg} for relevant definitions). 

\begin{theorem}\label{T:main}
Suppose $\mfG$ is a sub-Riemannian Carnot group. The group $\mfG$ admits a $1$-quasiconformal metric inversion if and only if it is isomorphic to a generalized Heisenberg group.
\end{theorem} 

While the assumption of $1$-quasiconformality is strong it is not altogether unnatural. Indeed, the inversions of \cite[Theorem 5.1]{CDKR-Heisenberg} are $1$-quasiconformal. Furthermore, \rf{T:main} illustrates a special case of the characterization of conformal maps on Carnot groups recently provided by \cite[Theorem 4.1]{CO14}. One should note, however, that the methods used to prove \rf{T:main} are somewhat different from those implemented in \cite{CO14}. 

In general, metric inversions are $K$-quasiconformal, with $1\leq K<+\infty$ (see \rf{S:defs}). One can obtain an analogue of \rf{T:main} for general metric inversions if more information is assumed about the Carnot groups in question. In particular, the following fact is contained in recent work of Xiangdong Xie (see \cite{Xie-non-rigid}, \cite{Xie-rigidity}, and \cite{Xie-filiform} for relevant definitions and results). 

\begin{fact}\label{F:xie}
Suppose $\mfG$ is a non-rigid sub-Riemannian Carnot group. If $\mfG$ is not contained in one of the following three classes of groups, then it does not admit a metric inversion:
\begin{enumerate}
  \item{Euclidean groups}
  \item{Heisenberg product groups}
  \item{Complex Heisenberg product groups}
\end{enumerate}
\end{fact}

Let $\mfG$ be a non-rigid Carnot group that is not contained in one of the three classes listed above. Let $\hat\mfG$ denote a metric sphericalization of $\mfG$ (see \rf{S:defs}). If follows from the aforementioned work of Xie that any quasiconformal self-homeomorphism $f:\hat\mfG\to\hat\mfG$ must permute left cosets of certain proper subgroups of $\mfG$, and therefore must fix the point at infinity. \rf{F:xie} then follows from the observation that metric inversions do not fix the point at infinity.  

\medskip
While \rf{T:main} falls within the general study of inversion invariant bi-Lipschitz homogeneity, it exhibits a special case of this property in the particular setting of Carnot groups. Note that (non-abelian) Carnot groups are \textit{fractal} in the sense that their topological and Hausdorff dimensions do not agree. When we restrict ourselves to non-fractal metric spaces, we can use recent work of Kinneburg (see \cite{Kinneberg-fractal}) to obtain the following result. 

\begin{theorem}\label{T:euclidean}
Suppose $X$ is a proper and connected metric space whose Hausdorff and topological dimensions are both equal to $n\in\mfN$. The space $X$ is inversion invariant bi-Lipschitz homogeneous if and only if it is bi-Lipschitz equivalent to $\mfR^n$ (when $X$ is unbounded) or $\mfS^n$ (when $X$ is bounded). 
\end{theorem}

In fact, an analogue of \rf{T:euclidean} holds when the Hausdorff dimension of $X$ is strictly larger than its topological dimension, but this requires additional assumptions about the one-dimensional metric structure of $X$ (cf.\,\cite[Theorem 1.5]{Kinneberg-fractal}). 

\begin{acknowledgments}
The proof of \rf{T:main} was inspired by observations made in \cite[Section 6.3]{BS13} about the Tits classification of $2$-transitive group actions. The author is grateful to Xiangdong Xie for many helpful discussions during the preparation of this paper, and for the helpful suggestions of the anonymous referees.

In this revised version of the paper, we replace Lemma 4.3 (in the indexing of the previous version) with a clarified proof of Theorem 1.1. 
\end{acknowledgments}

\section{Notation and Definitions}\label{S:defs}     

Given two numbers $A$ and $B$, we write $A\eqx_C B$ to indicate that $C^{-1}A\leq B\leq CA$, where $C$ is  independent of $A$ and $B$. When the quantity $C$ is understood or irrelevant, we simply write $A\eqx B$.

Given a metric space $(X,d)$, $r>0$, and $x\in X$ we write $B(x;r):=\{y\in X:d(x,y)<r\}$ to denote an open ball in $X$ centered at $x$ of radius $r$. A metric space $X$ is said to be \textit{proper} if the closures of open balls in $X$ are compact.

Given $1\leq L<+\infty$, an embedding $f:X\to Y$ is \textit{$L$-bi-Lipschitz} provided that for all points $x,y\in X$ we have $d_Y(f(x),f(y))\eqx_L d_X(x,y)$. Two spaces $X$ and $Y$ are \textit{$L$-bi-Lipschitz equivalent} if there exists an $L$-bi-Lipschitz homeomorphism between the two spaces. A space $X$ is \textit{bi-Lipschitz homogeneous} if there exists a collection $\mathcal{F}$ of bi-Lipschitz self-homeomorphisms of $X$ such that, for every pair $x,y\in X$, there exists $f\in \mathcal{F}$ with $f(x)=y$. When every map in $\mathcal{F}$ is $L$-bi-Lipschitz we say that $X$ is $L$-bi-Lipschitz homogeneous, or \textit{uniformly bi-Lipschitz homogeneous} when the particular distortion bound is not important. 

An embedding $f:X\to Y$ is \textit{$\theta$-quasim\"obius} if $\theta:[0,+\infty)\to[0,+\infty)$ is a homeomorphism and, for all quadruples $x,y,z,w$ of distinct points in $X$, we have
\[\frac{d(f(x),f(y))d(f(z),f(w))}{d(f(x),f(z))d(f(y),f(w))}\leq \theta\left(\frac{d(x,y)d(z,w)}{d(x,z)d(y,w)}\right).\]
When there exists a constant $1\leq C<+\infty$ such that $\theta(t)=Ct$, we say that $f$ is \textit{strongly quasim\"obius}. See \cite{Kinneberg-fractal} for a detailed study of such mappings.

An embedding $f:X\to Y$ is \textit{metrically $K$-quasiconformal} if $1\leq K<+\infty$ and, for all $x\in X$,
\[\limsup_{r\to0}\frac{\sup\{d(f(x),f(y)):d(x,y)\leq r\}}{\inf\{d(f(x),f(z)):d(x,z)\geq r\}}\leq K.\] 
In the present paper, since no other definition of quasiconformality is used, we drop the qualifier `metrically,' and simply refer to $K$-quasiconformal mappings. 

\medskip
Given an unbounded metric space $(X,d)$ and a point $p\in X$, we define $\hat{X}:=X\cup\{\infty\}$ and $X_p:=X\setminus\{p\}$. We say that a homeomorphism $\phi:X_p\to X_p$ is an \textit{$L$-metric inversion} provided that there exists a constant $1\leq L<+\infty$ such that for any $x,y\in X_p$, 
\[d(\phi(x),\phi(y))\eqx_L\frac{d(x,y)}{d(x,p)d(y,p)}.\]
We extend $\phi$ to $\hat{X}$ by the definitions $\phi(\infty):=p$ and $\phi(p):=\infty$. We say that \textit{$(X,d)$ admits a metric inversion} provided that there exists a point $p\in X$ such that an $L$-metric inversion exists on $X_p$. We note that the inversions of \cite[Theorem 4.2]{CDKR-Heisenberg} are $1$-metric inversions. In general, while $L$-metric inversions are $L^2$-quasiconformal, they need not be $1$-quasiconformal. 

The definition of metric inversion given above is closely related to the metric inversion described in \cite{BHX-inversions}. In \cite{BHX-inversions}, the authors construct a distance $d_p$ on $\hat{X}_p$ such that, for any $x,y\in X_p$,
\[\frac{1}{4}\cdot\frac{d(x,y)}{d(x,p)d(y,p)}\leq d_p(x,y)\leq\frac{d(x,y)}{d(x,p)d(y,p)},\]
with obvious analogues when $x$ or $y$ is equal to $\infty$. They also prove that the identity map from $(X_p,d)$ to $(X_p,d_p)$ is $16t$-quasim\"obius (\cite[Lemma 3.1]{BHX-inversions}) and 1-quasiconformal at non-isolated points (\cite[Proposition 4.1]{BHX-inversions}). While the \cite{BHX-inversions} definition of metric inversion is valid for both bounded and unbounded spaces, if $X$ is an unbounded space admitting an $L$-metric inversion $\phi$ at $p$, then $\phi:(X_p,d_p)\to(X_p,d)$ is a $4L$-bi-Lipschitz homeomorphism. 

A related concept is that of \textit{metric sphericalization}. Again following \cite{BHX-inversions}, the metric sphericalization of an unbounded metric space $(X,d)$ at some basepoint $p\in X$ is denoted by $(\hat{X},\hat{d}_p)$. Here $\hat{d}_p$ is a distance such that, for any $x,y\in \hat{X}$,
\begin{equation}\label{E:sphere}
\frac{1}{4}\cdot\frac{d(x,y)}{(1+d(x,p))(1+d(y,p))}\leq\hat{d}_p(x,y)\leq\frac{d(x,y)}{(1+d(x,p))(1+d(y,p))}.
\end{equation}
As with metric inversion (see \cite[Section 3.B]{BHX-inversions}), the identity map from $(X,d)$ to $(\hat{X},\hat{d}_p)$ is $16t$-quasim\"obius and $1$-quasiconformal at non-isolated points. 

Following \cite{Freeman-iiblh} and \cite{Freeman-params}, a metric space $(X,d)$ is \textit{inversion invariant bi-Lipschitz homogeneous} provided that both $(X,d)$ and $(\hat{X}_p,d_p)$ are uniformly bi-Lipschitz homogeneous. This definition is independent of $p\in X$, up to a quantitative change in parameters. One can verify that, when $(X,d)$ is unbounded, this definition is equivalent to the statement that $(X,d)$ is bi-Lipschitz homogeneous and $(X,d)$ admits a metric inversion.

Inversion invariant bi-Lipschitz homogeneous metric spaces are often \textit{Ahlfors $Q$-regular} (see \cite[Theorem 1.1]{Freeman-iiblh}). Given $0< Q<+\infty$, a metric space $(X,d)$ with Borel measure $\mu$ is Ahlfors $Q$-regular provided that there exists a constant $1\leq C<+\infty$ such that $\mu(B(x;r))\eqx_C r^Q$ for every $x\in X$ and $0<r\leq\diam(X)$. 

\medskip
\textit{Carnot groups} are examples of Ahlfors $Q$-regular metric spaces. A Carnot group $\mathbb{G}$ of step $n\in\mfN$ is a connected, simply connected, nilpotent Lie group with stratified Lie algebra $\Lie(\mathbb{G})=V_1\oplus V_2\oplus\dots\oplus V_n$. The \textit{layers} $V_i$ are such that, for $1\leq j\leq n-1$, we have $[V_j,V_1]=V_{j+1}$. Here $[X,Y]=XY-YX$ denotes the Lie bracket. We require $V_n\not=\{0\}$, and that for each $1\leq j\leq n$ we have $[V_j,V_{n}]=\{0\}$. We refer to $V_1$ as the \textit{horizontal layer} of $\Lie(\mathbb{G})$. By left-translation $V_1$ is extended to a (left-invariant) distribution $\Delta$ on $\mathbb{G}$, referred to as the \textit{horizontal distribution}.

When $\Delta$ is equipped with a left-invariant norm $|\cdot|$, we define the associated \textit{sub-Finsler distance $d_{SF}$} on $\mathbb{G}$ as follows. Let $\gamma:[0,1]\to\mathbb{G}$ be an absolutely continuous path. The path $\gamma$ is \textit{horizontal} provided that for almost every $t\in[0,1]$ we have $\dot\gamma(t)\in \Delta$.  The $d_{SF}$ length of a horizontal path $\gamma$ is $\ell_{SF}(\gamma):=\int_0^1|\dot\gamma(t)|dt$.
We then define 
\begin{equation*}\label{E:cc-dist}
d_{SF}(x,y):=\inf\{\ell_{SF}(\gamma):\,\gamma \text{ a horizontal path such that } \gamma(0)=x,\,\gamma(1)=y\}.
\end{equation*}
By well known results of Chow and Rashevskii, $d_{SF}$ defines a geodesic distance on $\mathbb{G}$. Thus a \textit{sub-Finsler Carnot group} is a Carnot group equipped with a distance $d_{SF}$. When the norm on $\Delta$ is derived from an inner product, we obtain a \textit{sub-Riemannian distance} on $\mfG$, denoted by $d_{SR}$. 

Since the norm on $\Delta$ is left-invariant, for any element $g\in \mfG$, the left translation $\ell_g(x):=gx$ is an isometry of $(\mfG,d_{SF})$. Sub-Finsler distance is also homogeneous with respect to \textit{canonical dilations}. For an element $x\in \mfG$, let $X:=\log(x)=\sum_{i=1}^nX_i$, where $X_i\in V_i$. For any $t>0$, define the canonical dilation $\delta_t:(\mfG,d_{SF})\to(\mfG,d_{SF})$ as $\delta_t(x)=\exp\left(\sum_{i=1}^nt^iX_i\right)$. For points $x,y\in\mfG$, we have $d_{SF}(\delta_t(x),\delta_t(y))=t\,d_{SF}(x,y)$.

\section{Generalized Heisenberg Groups}\label{S:heisenberg}    

Let $\mathfrak{n}$ denote a Lie algebra endowed with an inner product $\langle\cdot,\cdot\rangle$ and accompanying norm $|\cdot|$. Suppose that $\mathfrak{n}$ is either abelian or stratified of step two. In the step two case, this means that there exist non-trivial complementary orthogonal subspaces $\mathfrak{v}$ and $\mathfrak{z}$ such that $[\mathfrak{v},\mathfrak{v}]=\mathfrak{z}$ and $[\mathfrak{v},\mathfrak{z}]=\{0\}=[\mathfrak{z},\mathfrak{z}]$. For $X,Y\in\mathfrak{v}$ and $Z\in\mathfrak{z}$, let the map $J:\mathfrak{z}\to\End(\mathfrak{v})$ be defined via the formula $\langle J_ZX,Y\rangle=\langle Z,[X,Y]\rangle$. The algebra $\mathfrak{n}$ is of \textit{Heisenberg type} provided that, for all $X\in\mathfrak{v}$ and $Z\in\mathfrak{z}$, we have $|J_ZX|=|Z||X|$. Equivalently, $J_Z^2=-|Z|^2I$, where $I$ denotes the identity map. Various properties of the map $J$ are documented in \cite[Section 2(a)]{CDKR98}. A simply connected Lie group is said to be of Heisenberg type if its Lie algebra is of Heisenberg type.

Given a Lie algebra $\mathfrak{n}$ of Heisenberg type, we say that $\mathfrak{n}$ satisfies the \textit{$J^2$-condition} provided that, for any $X\in\mathfrak{v}$ and any two orthogonal elements $Z,Z'\in\mathfrak{z}$, there exists some element $Z''\in\mathfrak{z}$ such that $J_Z J_{Z'}X=J_{Z''}X$. Note that if $\dim(\mathfrak{z})\in\{0,1\}$, this condition is vacuously satisfied. 

\medskip
We say that a Carnot group is a \textit{generalized Heisenberg group} if it is a Heisenberg group over $\mfK$, where (here and in the sequel) $\mfK$ denotes either the real numbers $\mfR$, complex numbers $\mfC$, quaternions $\mfH$, or octonians $\mfO$. These groups are defined as follows:

\begin{itemize}
  \item{The Heisenberg group over $\mfR$, or the \textit{real Heisenberg group} $H_\mfR$, is $\mfR^n$.}
  \item{The Heisenberg group over $\mfC$, or the \textit{complex Heisenberg group} $H_\mfC$, is the Carnot group with step two real Lie algebra $\mathfrak{n}=\mathfrak{v}\oplus \mathfrak{z}$, where $\mathfrak{v}:=\Span\{X_i,Y_i:1\leq i\leq n\}$ and $\mathfrak{z}:=\Span\{Z\}$. Equip $\mathfrak{n}$ with an inner product such that $\{X_i,Y_i,Z:1\leq i\leq n\}$ is an orthonormal basis. The only non-trivial bracket relations are $[X_i,Y_i]=Z$, for $1\leq i\leq n$.} 
  \item{The Heisenberg group over $\mfH$, or the \textit{quaternionic Heisenberg group} $H_\mfH$, is the Carnot group with step two real Lie algebra $\mathfrak{n}=\mathfrak{v}\oplus\mathfrak{z}$, where $\mathfrak{v}=\Span\{X_i,Y_i,V_i,W_i:1\leq i \leq n\}$ and $\mathfrak{z}=\Span\{Z_k:1\leq k \leq 3\}$. Equip $\mathfrak{n}$ with an inner product such that $\{X_i,Y_i,V_i,W_i,Z_k:1\leq i\leq n, 1\leq k \leq 3\}$ is an orthonormal basis. For $1\leq i \leq n$, the only nontrivial bracket relations are $[X_i,Y_i]=Z_1=[V_i,W_i]$, $[X_i,V_i]=Z_2=[W_i,Y_i]$, and $[X_i,W_i]=Z_3=[Y_i,V_i]$.}
  \item{The Heisenberg group over $\mfO$, or the \textit{octonionic Heisenberg group} $H_\mfO$, is the Carnot group with step two real Lie algebra $\mathfrak{n}=\mathfrak{v}\oplus\mathfrak{z}$, where $\mathfrak{v}=\Span\{X_i:0\leq i\leq 7\}$ and $\mathfrak{z}=\Span\{Z_k:1\leq k \leq 7\}$. Equip $\mathfrak{n}$ with an inner product such that $\{X_i,Z_k:0\leq i\leq 7, 1\leq k \leq 7\}$ is an orthonormal basis. The only nontrivial bracket relations are $[X_0,X_k]=Z_k$ for $1\leq k\leq 7$ and $[X_i,X_j]=\varepsilon_{ijk}Z_k$, for $1\leq i,j,k\leq 7$. Here $\varepsilon$ is a completely antisymmetric tensor whose value is $+1$ when $ijk=124, 137, 156, 235, 267, 346, 457$.}
\end{itemize}

Via exponential coordinates, parameterizations of the groups $H_\mfK$ can be obtained as follows:
\begin{itemize}
  \item{When $\mfK=\mfR$, the abelian group $H_\mfK$ is equal to $\mfR^n$.}
  \item{When $\mfK=\mfC$, for each $1\leq i\leq n$, identify $x_iX_i+y_iY_i$ with $x_ie_0+y_ie_1\in\mfC$. Here $\{e_0,e_1\}$ is the canonical basis over $\mfR$ for $\mfC$. Thus $\Span\{X_i,Y_i:1\leq i\leq n\}$ is identified with $\mfC^n$. Identify $zZ$ with $ze_1\in\Im(\mfC)$.}
  \item{When $\mfK=\mfH$, for each $1\leq i\leq n$, identify $x_iX_i+y_iY_i+v_iV_i+w_iW_i$ with $x_ie_0+y_ie_1+v_ie_2+w_ie_3$.   Here $\{e_i\}_{i=0}^3$ is the canonical basis over $\mfR$ for $\mfH$. Thus $\Span\{X_i,Y_i,V_i,W_i:1\leq i\leq n\}$ is identified with $\mfH^n$. For each $1\leq k\leq 3$, identify $z_kZ_k$ with $z_ke_k\in\Im(\mfH)$.}
  \item{When $\mfK=\mfO$, identify $\sum_{i=0}^7x_iX_i$ with $\sum_{i=0}^7x_ie_i$ and $\sum_{k=1}^7z_kZ_k$ with $\sum_{k=1}^7z_ke_k$. Here $\{e_i\}_{i=0}^7$ is the canonical basis over $\mfR$ for $\mfO$. Thus $\Span\{X_i:0\leq i\leq7\}$ is identified with $\mfO$ and $\Span\{Z_k:1\leq k \leq 7\}$ is identified with $\Im(\mfO)$.}
\end{itemize} 

Extending the above identifications by linearity allows us to parameterize $H_\mfK$ by $\mfK^n\oplus \Im(\mfK)$. When $\mfK=\mfR$ we have $\Im(\mfK)=\{0\}$, and when $\mfK=\mfO$ we have $n=1$. Let $(x,z)=(\sum_{i=1}^nx_i,z)$ denote a point in $H_\mfK=\mfK^n\oplus\Im(\mfK)$. Via the Baker-Campbell-Hausdorff formula, for points $(x,z),(x',z')\in H_\mfK$, the group law reads as
\[(x,z)(x',z')=\left(x+x',z+z'-\frac{1}{2}\sum_{i=1}^n\Im(x_i\overline{x_i}')\right).\]

Let $m=\dim\left(\Im(\mfK)\right)$, so that $m\in\{0,1,3,7\}$. Given a canonical basis element $e_k\in \Im(\mfK)$,  we define a map $L_k\in\End(\mfK^n)$ such that, for $x\in\mfK^n$, $L_k(x)=e_kx=\sum_{i=0}^ne_kx_i$. In other words, $L_k$ denotes left multiplication in $\mfK^n$ by $e_k$. Passing through the above parameterization of $H_\mfK$ and extending by linearity, this gives rise to a map $J:\mathfrak{z}\to\End(\mathfrak{v})$ such that, for any $X,Y\in\mathfrak{v}$ and $Z\in\mathfrak{z}$, we have $\langle J_ZX,Y\rangle=\langle[X,Y],Z\rangle$. Furthermore, it is straightforward to verify that, for any $Z\in \mathfrak{z}$, the map $J_Z:\mathfrak{v}\to\mathfrak{v}$ satisfies $J_Z^2=-|Z|^2I$. For two orthogonal elements $Z,Z'\in\mathfrak{z}$, there exists $Z''\in\mathfrak{z}$ such that $J_ZJ_{Z'}=J_{Z''}$. When $\mfK=\mfH$, this last observation follows from the associativity of left multiplication. In the case that $\mfK=\mfO$, this follows from the fact that $\dim(\mathfrak{z})=\dim(\mathfrak{v})-1$.  Thus we verify that generalized Heisenberg groups are of Heisenberg type and satisfy the $J^2$-condition.

For any Lie algebra $\mathfrak{n}$ and corresponding simply connected Lie group $N$ of Heisenberg type, one may construct the Lie algebra $\mathfrak{s}:=\mathfrak{n}\oplus\mathfrak{a}$, where $\mathfrak{a}$ is a one-dimensional Lie algebra with inner product. Let $\mathfrak{a}$ be spanned by the unit vector $H$. The Lie bracket on $\mathfrak{s}$ is determined by the requirements that $[H,X]=\frac{1}{2}X$ and $[H,Z]=Z$ for any $X\in \mathfrak{v}$ and any $Z\in\mathfrak{z}$. We extend the inner products on $\mathfrak{n}$ and $\mathfrak{a}$ to $\mathfrak{s}$ by requiring that $\mathfrak{n}$ is orthogonal to $\mathfrak{a}$. Proceeding as in \cite[Section 3(a)]{CDKR98}, one obtains the group $S:=\exp(\mathfrak{s})$ as a semidirect product $NA$, where $A:=\exp(\mathfrak{a})$. If we parameterize $S$ via $\mathfrak{v}\times\mathfrak{z}\times\mfR_+$ by identifying $(X,Z,t)$ with $\exp(X+Z)\exp(\log(t)H)\in S$, then an element $a_t=(0,0,t)\in A\subset S$ acts on $n=(X,Z,1)\in N\subset S$ by $a_t(n)=(t^{1/2}X,tZ,t)$. By translating the inner product on $\mathfrak{s}$, we obtain a left invariant distance on $S$.

One can then proceed to construct the Siegel-type domain 
\[D:=\left\{(X,Z,t):t-\frac{|X|^2}{4}>0\right\}\subset\mathfrak{v}\oplus\mathfrak{z}\oplus\mathfrak{a}.\]
The domain $D$ can be explicitly identified with $S$. When the left-invariant distance on $S$ is pulled back to $D$ via this identification, $S$ possesses a simply transitive action on $D$ by affine transformations (see \cite[Section 3(b)]{CDKR98}). The group $N$ can be identified with the set 
\[\partial D=\left\{\left(X,Z,\frac{|X|^2}{4}\right)\right\}\subset\mathfrak{v}\oplus\mathfrak{z}\oplus\mathfrak{a}.\]

Given $n=(x,z)=\exp(X+Z)\in N$, write 
\begin{equation}\label{E:gauge}
\|n\|:=\left(\frac{|X|^4}{16}+|Z|^2\right)^{1/4}
\end{equation}
The function $d_N(n,n'):=\|n'^{-1}n\|$ defines a distance on $N$ that is invariant under left multiplication. Furthemore, the action of $A$ extends to $\partial D$ such that for any $n,n'\in N$ and $a_t\in A$, we have $d_N(a_t(n),a_t(n'))=t^{1/2}d_N(n,n')$.  

When $N=H_\mfK$, the geometry of $D$ is identified by the following result (\cite{CDKR-Heisenberg}, \cite{CDKR98}, \cite[Theorem 4.1.9.A]{Vanhecke95}). 

\begin{fact}\label{F:classification}
Suppose that $D$ is the Siegel-type domain associated to a Lie algebra $\mathfrak{n}=\mathfrak{v}\oplus\mathfrak{z}$ of Heisenberg type. Let $m:=\dim(\mathfrak{z})$. The algebra $\mathfrak{n}$ satisfies the $J^2$-condition if and only if

\parbox[1cm]{\textwidth}{
\begin{enumerate}
  \item[($m=0$)]{the space $D$ is isometric to real hyperbolic space.}
  \item[($m=1$)]{the space $D$ is isometric to complex hyperbolic space.}
  \item[($m=3$)]{the space $D$ is isometric to quaternionic hyperbolic space.}
  \item[($m=7$)]{the space $D$ is isometric to octonionic hyperbolic space.}
\end{enumerate}} 
In particular, $\mathfrak{n}$ satisfies the $J^2$-condition if and only if $\exp(\mathfrak{n})$ is isomorphic to $H_\mfK$. 
\end{fact}

Assume $N=H_\mfK$, and let $G_\mfK$ denote the isometry group of $D$. Then $H_\mfK AK$ is an Iwasawa decomposition of $G_\mfK$, where $K$ is the stabilizer of $(0,0,1)\in D$ and $A$ is as above. Writing $M$ to denote the centralizer of $A$ in $K$, \cite[Theorem 7.4]{CDKR98} provides the Bruhat decomposition $G_\mfK=(H_\mfK AM)\cup(H_\mfK \sigma H_\mfK AM)$. Here $\sigma$ is the geodesic inversion of $D$ described in \cite[Section 3(c)]{CDKR98}. 

Via an appropriate Cayley transformation $C$, the domain $D$ can be identified with the unit ball $B\subset\mathfrak{v}\oplus\mathfrak{z}\oplus\mathfrak{a}$. This Cayley transformation $C:D\to B$ can be continuously extended to a homeomorphism between the one point compactification of $D\cup\partial D$ and the closed unit ball $\overline{B}$, where $C(0,0,0)=(0,0,1)$ and $C(\infty)=(0,0,-1)$. Via this identification, the action of $G_\mfK$ can be continuously extended to $\overline{B}$. The stabilizer of $(0,0,-1)\in\partial{B}$ is $H_\mfK AM$ (see the proof of \cite[Theorem 7.4]{CDKR98}). The stabilizer of both $(0,0,1)$ and $(0,0,-1)$ is $AM$. In this way we can view $G_\mfK$ as a group acting on $\hat H_\mfK=\partial B$. The subgroup $H_\mfK AM$ fixes the point at infinity $\infty\in\hat H_\mfK$, and the subgroup $AM$ fixes both the identity element $e\in H_\mfK$ and $\infty\in \hat H_\mfK$.

\section{Preliminary Facts and Lemmas}\label{S:prelims}     

The following fact is a special case of \cite[Theorem 3.3]{Kramer-transitive}. 

\begin{fact}\label{F:two-transitive}
Let $G$ denote a group acting effectively and 2-transitively on a topological sphere. If $G$ is locally compact and $\sigma$-compact, then the identity component of $G$ is a simple Lie group isomorphic to either $G_\mfK$ (when $\mfK=\mfC$, $\mfH$, or $\mfO$) or to an index two subgroup of $G_\mfK$ (when $\mfK=\mfR$).
\end{fact}

The following technical lemma will allow us to apply \rf{F:two-transitive} to prove \rf{T:main}. 

\begin{lemma}\label{L:locally-compact}
Suppose $X$ is a compact metric space and $G\subset\mathcal{C}(X,X)$ is a group of uniformly quasim\"obius self-homeomorphisms. The closure of $G$ is locally compact and $\sigma$-compact in the topology of uniform convergence. 
\end{lemma}

\begin{proof}
Let $X^3$ denote the space of ordered triples from $X$ endowed with the product distance. Since $X$ is compact, $X^3$ is separable. Let $\{T_i\}$ denote a countable dense subset of $X^3$, where $T_i=(x_i,y_i,z_i)$. Given $T_i$, $T_j$, and $\varepsilon>0$, define 
\[B_{i,j}(\varepsilon):=\{g\in G:g(T_i)\subset N(T_j;\varepsilon)\},\]
where, for $E\subset X$ and $r>0$, $N(E;r):=\{x\in X:\dist(x,E)<r\}$. The sets $B_{i,j}(\varepsilon)$ are open in the compact-open topology (and hence in the uniform convergence topology). Furthermore, for fixed $i,j$ there exists $\varepsilon_{i,j}>0$ such that the set $B_{i,j}(\varepsilon_{i,j})$ is equicontinuous  (see \cite[Theorem 2.1]{Vaisala-qm}). In fact, one can take $\varepsilon_{i,j}:=\sep(T_j)/4$, where $\sep(T_j)$ denotes the minimal distance between distinct pairs in $T_j$. Therefore, by Arzela-Ascoli, $B_{i,j}(\varepsilon_{i,j})$ has compact closure. One can check that $G\subset\cup_{i,j}B_{i,j}(\varepsilon_{i,j})$, and so we conclude that the closure of $G$ is locally compact. Since the collection $\{B_{i,j}(\varepsilon_{i,j})\}$ is countable, the closure of $G$ is $\sigma$-compact. 
\end{proof}

\section{Proofs of \rf{T:main} and \rf{T:euclidean}}\label{S:proofs}  

\begin{proof}[Proof of \rf{T:main}] For use below, we begin by confirming that the composition of (metrically) $1$-quasiconformal maps between open sets of a sub-Riemannian Carnot group remains $1$-quasiconformal. To this end, let $f_1:\Omega_1\to\Omega_2$ and $f_2:\Omega_2\to\Omega_3$ denote $1$-quasiconformal homeomorphisms between open sets of a Carnot group $\mfG$. By \cite[Corollary 7.1]{CC06}, both $f_1$ and $f_2$ are smooth and Pansu differentiable everywhere in their domains. Furthermore, the Lie derivatives of their Pansu differentials are similarities at every point in their domains (when restricted to the horizontal layer). By \cite[Lemma 3.7]{CC06}, these properties are invariant under compositions. Therefore, (again using \cite[Corollary 7.1]{CC06}) we conclude that $f_2\circ f_1:\Omega_1\to\Omega_3$ remains $1$-quasiconformal. 

In order to proceed with the proof, we invoke the assumption that, for some $p\in\mfG$, there exists a $1$-quasiconformal $L$-metric inversion $\phi:(\mfG_p,d_{SR})\to(\mfG_p,d_{SR})$. For any $x\in\mfG$, write $\phi_x:=\ell_{xp^{-1}}\circ\phi\circ\ell_{xp^{-1}}^{-1}$. Since left translations are isometries, the map $\phi_x:(\mfG_x,d_{SR})\to(\mfG_x,d_{SR})$ is a $1$-quasiconformal $L$-metric inversion such that $\phi_x(x)=\infty$ and $\phi_x(\infty)=x$. 

Let $\{x,x',y,y'\}$ denote a quadruple of points in $\hat{\mfG}$ such that $x'=y'$ if and only if $x=y$. We consider the two possible cases for such a quadruple: 

\begin{itemize}
  \item{Assume that $x\not= y$ (and so $x'\not=y'$). If all points are finite, define the map $g:=\phi_{x'}\circ \ell_z\circ \phi_x:\hat\mfG\to\hat\mfG$, where $z:=\phi_{x'}^{-1}(y')\phi_x(y)^{-1}\in\mfG$. If $x=\infty$ and/or $x'=\infty$, then replace $\phi_x$ and/or $\phi_{x'}$, respectively, with the identity map $\id:\hat\mfG\to\hat\mfG$.}
  \item{Assume that $x=y$ (and so $x'=y'$). If both points $x$ and $x'$ are finite, then replace $\phi_x$ and $\phi_{x'}$ with the identity map on $\mfG$. If $x$ and/or $x'$ is the point at infinity, then replace $\phi_x$ and/or $\phi_{x'}$, respectively, with the identity map and replace $\ell_z$ with the identity map.} 
\end{itemize} 

In either case we have $g(x)=x'$ and $g(y)=y'$. Let $G_*$ denote the group generated by finite compositions of the maps $g$ as constructed above. By construction, $G_*$ acts two-transitively on $\hat\mfG$. For each $h\in G_*$, if $h$ fixes the point at infinity, then by the initial paragraph of the proof, $h:(\mfG,d_{SR})\to(\mfG,d_{SR})$ is $1$-quasiconformal away from at most a finite set. In the case that $h(\infty)\not=\infty$ we obtain the same conclusion for $h:(\mfG\setminus h^{-1}(\infty),d_{SR})\to(\mfG\setminus h(\infty),d_{SR})$.

Fix $h\in G_*$. If $h(\infty)\not=\infty$, then define $h_\infty:=\phi_{h(\infty)}\circ h$. If $h(\infty)=\infty$, then define $h_\infty:=h$. In either case, the map $h_\infty:(\mfG,d_{SR})\to(\mfG,d_{SR})$ is a homeomorphism, and (as noted in the preceding paragraph) $h_\infty$ is $1$-quasiconformal away from a finite set. By \cite[Theorem 5.2]{BKR-ACQC}, $h_\infty:(\mfG,d_{SR})\to(\mfG,d_{SR})$ is weakly $H$-quasisymmetric (in the sense of \cite{Vaisala-cylinders}), where $H$ depends only on $\mfG$. Therefore, by \cite[Theorem 2.9]{Vaisala-cylinders}, we conclude that $h_\infty:(\mfG,d_{SR})\to(\mfG,d_{SR})$ is $\eta$-quasisymmetric, with $\eta$ depending only on $\mfG$. By \cite[Theorem 3.2]{Vaisala-qm}, $h_\infty:(\mfG,d_{SR})\to(\mfG,d_{SR})$ is also $\theta''$-quasimobius, with $\theta''$ depending only on $\mfG$. By \cite[Section 3.B]{BHX-inversions}, the identity map between $(\mfG,d_{SR})$ and the sphericalized space $(\mfG,\hat{d}_{SR})$ is $16t$-quasim\"obius. It then follows that $h_\infty:(\hat\mfG,\hat{d}_{SR})\to(\hat{\mfG},\hat{d}_{SR})$ is $\theta'$-quasim\"obius, with $\theta'$ determined solely by $\theta''$. Since $h_\infty=\phi_{h(\infty)}\circ h$, and $\phi_{h(\infty)}:(\hat{\mfG},\hat{d}_{SR})\to(\hat{\mfG},\hat{d}_{SR})$ is a $256L^4t$-quasim\"obius homeomorphism, it follows that $h:(\hat{\mfG},\hat{d}_{SR})\to(\hat{\mfG},\hat{d}_{SR})$ is $\theta$-quasim\"obius, where $\theta$ is determined solely by $L$ and $\mfG$. In conclusion, the group $G_*$ consists of uniformly quasim\"obius self-homeomorphisms of $(\hat{\mfG},\hat{d}_{SR})$. 

Note that canonical dilations of $\mfG$ were not included in the generating set for $G_*$. However, if $\Gamma$ denotes the group of canonical dilations of $\mfG$, then the closure of $\langle G_*,\Gamma\rangle$ is isomorphic to $G$. In particular, we may assume that $\Gamma\subset G$. This follows from \rf{L:locally-compact} and \cite[Theorem 3.3]{Kramer-transitive} because both $\langle G_*,\Gamma\rangle$ and $G_*$ act effectively and $2$-transitively on $(\hat\mfG,\hat{d}_{SR})$ by uniformly quasim\"obius homeomorphisms. 

Given a topological group $H$, let $H^\circ$ denote the identity component. Via \rf{F:two-transitive} we conclude that $G^\circ$ is isomorphic to one of $G_\mfR^\circ$, $G_\mfC$, $G_\mfH$, or $G_\mfO$. In short, $G^\circ$ is isomorphic to $G_\mfK^\circ$. More precisely (see \cite[Theorem 3.3 and Proposition 7.1]{Kramer-transitive}), there exists an isomorphism $\psi:G^\circ\to G_\mfK^\circ$ and a homeomorphism $F:(\hat\mfG,\hat{d}_{SR})\to(\hat H_\mfK,\hat{d}_H)$ such that, for any $g\in G^\circ$ and $x\in \hat\mfG$, we have $F(g(x))=\psi(g)(F(x))$. Here $d_H$ is defined by \rf{E:gauge}, and $\hat{d}_H$ denotes the sphericalized distance $\widehat{(d_H)_e}$. Since $G_\mfK^\circ$ acts two-transitively on $(\hat{H}_\mfK,\hat{d}_H)$, we may assume $F(e)=e$ and $F(\infty)=\infty$.

Given $t>0$, the map $\psi(\delta_t)\in G_\mfK^\circ$ fixes the set $\{e,\infty\}$. Therefore, $\psi(\delta_t)\in AM$. It follows that there exists $a_{s(t)}\in A$ and $m_t\in M$ such that $\psi(\delta_t)=a_{s(t)}m_t$. Here $s:\mfR_+\to\mfR_+$ is a function such that, for any $s,r\in \mfR_+$, we have $s(rt)=s(r)s(t)$. Since, for any $x\in \mfG$, we have $F(\delta_r(x))=\psi(\delta_r)(F(x))$, it also follows that $\lim_{t\to+\infty}s(t)=+\infty$. Since $A$ commutes with $M$, we note that $\psi(\delta_t^{-1})=a_{1/s(t)}m_t^{-1}$.

Given $g\in\mfG$, the map $\psi(\ell_g)$ fixes the point at infinity in $\hat H_\mfK$. Therefore, $\psi(\ell_g)\in H_\mfK AM$, and there exist $h\in H_\mfK$, $a_r\in A$, and $m\in M$ such that $\psi(\ell_g)=\ell_h a_r m$. Combining this with the previous paragraph, we have $\psi(\delta_t^{-1}\ell_g\delta_t)=a_{1/s(t)}m_t^{-1}\ell_ha_rma_{s(t)}m_t$. We then observe that
\begin{align*}
a_{1/s(t)}m_t^{-1}\ell_ha_rma_{s(t)}m_t&=a_{r/s(t)}m_t^{-1}\ell_{[a_{1/r}(h)]}a_{s(t)}mm_t\\
&=a_rm_t^{-1}\ell_{[a_{1/(rs(t))}(h)]}mm_t
\end{align*}

Since $M$ is compact and $a_{1/(rs(t))}(h)\to e\in H_\mfK$ as $t\to+\infty$, there exists $m'\in M$ such that, up to a subsequence, 
\[a_{r}m_t^{-1}\ell_{[a_{1/(rs(t))}(h)]}mm_t\to a_rm'^{-1}mm'\]
as $t\to+\infty$. Here the convergence is uniform on compact subsets of $H_\mfK$. On the other hand, the map $\delta_t^{-1}\ell_g\delta_t=\ell_{[\delta^{-1}_t(g)]}$ is locally uniformly convergent to the identity map of $\mfG$. Since $\psi:G^\circ\to G^\circ_\mfK$ is an isomorphism and both groups act effectively, $a_rm'^{-1}mm'=\id$ as a map of $(H_\mfK,d_H)$. Via the Bruhat decomposition, this implies that $a_r=\id$ and $m'^{-1}mm'=\id$, and so $m=\id$. Therefore, $\psi(\ell_g)=\ell_h$. Since $\mfG$ acts simply transitively on itself by left translations, we conclude that $\psi(\mfG)=H_\mfK$. 

\medskip
In order to prove the reverse implication, suppose $\mfG$ is a generalized Heisenberg group equipped with a sub-Riemannian distance. Due to the comparability of the sub-Riemannian distance with the distance given by \rf{E:gauge}, by \cite[Theorem 4.2]{CDKR-Heisenberg}, the map $\sigma$ satisfies the definition of a metric inversion and is quasiconformal on $\mfG_e$. By (the proof of) \cite[Theorem 5.1]{CDKR-Heisenberg} the Riemannian differential of $\sigma$, when restricted to the horizontal distribution, is the composition of a dilation and an isometry at every point in $\mfG_e$. It then follows from \cite[Lemma 3.4 and Corollary 7.2]{CC06} that the map $\sigma$ is $1$-quasiconformal on $\mfG_e$.
\end{proof}

The above proof can be compared with the methods appearing in \cite{BS13}. In \cite{BS13}, the ideal boundaries of rank one symmetric spaces are characterized via the notion of \textit{space inversions}. Such inversions are M\"obius involutions that satisfy several additional properties related to the M\"obius structure of a metric space. While metric inversions on a bi-Lipschitz homogeneous space can be used to construct analogues to space inversions, such constructions need not possess all the properties required to apply the innovative techniques behind the proof of \cite[Theorem 1.1]{BS13}. 

\begin{proof}[Proof of \rf{T:euclidean}]
This result is obtained by combining results from \cite{BK-rigidity}, \cite{Kinneberg-fractal}, and \cite{Freeman-params}. We first consider the case that $(X,d)$ is unbounded. Suppose $(X,d)$ is inversion invariant bi-Lipschitz homogeneous with respect to a collection of uniformly $L$-bi-Lipschitz self-homeomorphisms $\mathcal{F}$ and an $L$-metric inversion $\phi$ based at some point $p\in X$. It follows from \cite[Theorem 2.7]{Freeman-params} that $(X,d)$ is doubling. Therefore, by \cite[Theorem 1.1]{Freeman-iiblh}, $(X,d)$ is Ahlfors $n$-regular. As noted in \cite[Fact 4.1]{Freeman-iiblh}, for any point $q\in X$, the sphericalized space $(\hat{X},\hat{d}_q)$ is also Ahlfors $n$-regular, with regularity constant depending only on the doubling and homogeneity constants for $(X,d)$.

Fix $q\in X$ and set $\delta:=(16L^3M(1+M))^{-1}$, where $1\leq M<+\infty$ is the quasi-self-similarity constant given by \cite[Theorem 2.7]{Freeman-params}. Note that $M$ depends only on $L$. Let $\{x_1,x_2,x_3\}$ denote a triple of distinct points in $\hat{X}$. We claim that there exists a strongly quasim\"obius homeomorphism $g:(\hat{X},\hat{d}_q)\to(\hat{X},\hat{d}_q)$ such that, for $i\not=j$, we have $\hat{d}_q(g(x_i),g(x_j))\geq\delta$. When $x_1\not=\infty$, there exists $f\in \mathcal{F}$ such that $f(x_1)=p$.  When $x_1=\infty$, set $f:=\phi$. In either case, $\phi\circ f(x_1)=\infty$ and we define $y:=\phi\circ f(x_2)$, $z:=\phi\circ f(x_3)$. Consider the sphericalized space $(\hat{X},\hat{d}_y)$. By \cite[Theorem 2.7]{Freeman-params}, there exists a map $h:(X,d)\to(X,d)$ such that $h(y)=y$ and, for all $u,v\in X$, we have $d(h(u),h(v))\eqx_MC\,d(u,v)$, where $C:=d(y,z)^{-1}$. It follows from \rf{E:sphere} that 
\[\hat{d}_y(h(y),h(z))\geq\frac{1}{4}\cdot\frac{d(h(y),h(z))}{1+d(h(y),h(z))}\geq\frac{1}{4M}\cdot\frac{C\,d(y,z)}{1+MC\,d(y,z)}=\frac{1}{4M(1+M)}\]
We also find that
\begin{align*}
\hat{d}_y(h(\infty),h(z))=\hat{d}_y(\infty,h(z))&\geq\frac{1}{4}\cdot\frac{1}{1+d(h(y),h(z))}\\
&\geq\frac{1}{4}\cdot\frac{1}{1+MC\,d(y,z)}=\frac{1}{4(1+M)}.
\end{align*}
Finally, we note that $\hat{d}_y(y,\infty)\geq1/4$. Therefore, $h\circ\phi\circ f$ maps $\{x_1,x_2,x_3\}$ to a $(4M(1+M))^{-1}$ separated triple in $(\hat{X},\hat{d}_y)$. By \cite[Lemma 3.2]{BHX-inversions}, there exists a $4L^3$-bi-Lipschitz homeomorphism $k:(\hat{X},\hat{d}_y)\to(\hat{X},\hat{d}_q)$. Therefore, for any triple of distinct points $\{x_1,x_2,x_3\}\subset (\hat{X},\hat{d}_q)$, there exists a map of the form $g=k\circ h\circ\phi\circ f$ such that $g$ is a $K t$-quasim\"obius self-homeomorphism of $(\hat{X},\hat{d}_q)$ that maps $\{x_1,x_2,x_3\}$ to a $\delta$-separated triple. Here $K $ depends only on $L$.  

By \cite[Theorem 5.1]{Kinneberg-fractal}, we conclude that $(\hat{X},\hat{d}_q)$ is strongly quasim\"obius equivalent to $\mfS^n$. To justify this application of \cite[Theorem 5.1]{Kinneberg-fractal}, note that this theorem follows from results in \cite[Section 5]{BK-rigidity}. These results are proved under the assumption of a group action by uniformly quasim\"obius homeomorphisms. However, as stated at the beginning of \cite[Section 5]{BK-rigidity}, these results also hold under the weaker assumption that triples can be uniformly separated by uniformly quasim\"obius maps. 

Since strongly quasim\"obius maps between bounded spaces are bi-Lipschitz (see \cite[Remark 3.2]{Kinneberg-fractal}), we find that $(\hat{X},\hat{d}_q)$ is bi-Lipschitz equivalent to $\mfS^n$. Finally, by \cite[Lemma 3.2 and Proposition 3.4]{BHX-inversions}, $(X,d)$ is bi-Lipschitz equivalent to $\mfR^n$. 

\medskip
To finish the proof, we consider the case that $(X,d)$ is bounded. Given any point $p\in X$, the space $(X_p,d_p)$ is unbounded and remains proper, connected, and inversion invariant bi-Lipschitz homogeneous. To verify that $(X_p,d_p)$ remains inversion invariant bi-Lipschitz homogeneous, note that the metric inversion of $(X_p,d_p)$ at any point $q\in X_p$ is bi-Lipschitz equivalent to $(X_q,d_q)$ via the identity map. We conclude as above that $(X_p,d_p)$ is bi-Lipschitz equivalent to $\mfR^n$. By \cite[Lemma 3.2 and Proposition 3.5]{BHX-inversions}, $(X,d)$ is bi-Lipschitz equivalent to $\mfS^n$.
\end{proof}

\bibliographystyle{amsalpha}  
\bibliography{bib}            

\end{document}

%% file: macros.tex
\newif\ifshowflag
\showflagtrue   
\newif\ifshowqstn
\showqstntrue   
\showqstnfalse 
\newif\ifshowinfo
\showinfotrue   
\showinfofalse 

\newcommand{\eqx}{\simeq}   

\renewcommand{\phi}{\varphi}

\DeclareMathOperator{\sep}{sep}
\DeclareMathOperator{\diam}{diam}
\DeclareMathOperator{\dist}{dist}

\DeclareMathOperator{\id}{id}

\newcommand{\mathfont}{\mathbb} 

\newcommand{\mfC}{{\mathfont C}}      
\newcommand{\mfH}{{\mathfont H}}      
\newcommand{\mfN}{{\mathfont N}}      
\newcommand{\mfR}{{\mathfont R}}      
\newcommand{\mfS}{{\mathfont S}}      
\newcommand{\mfK}{{\mathfont K}}      
\newcommand{\mfG}{{\mathfont G}}      
\newcommand{\mfO}{{\mathfont O}}      

\catcode`\@=11       

\def\rf#1{\@rf{#1}#1:;;}
\def\rfs#1{\@rfs{#1}#1:;;}
\def\rfm#1{\@rfF#1<>;;}

\def\@C{C}\def\@CC{CC}\def\@E{E}\def\@F{F}\def\@L{L}\def\@P{P}\def\@Q{Q}
\def\@R{R}\def\@S{S}\def\@T{T}\def\@TT{TT}\def\@X{X}\def\@s{s}\def\@f{f}\def\@H{H}

\def\@rf#1#2:#3;;{\def\@b{#2}
  \ifx\@b\@C Corollary~\ref{#1}\else%
  \ifx\@b\@CC Corollary~\ref{#1}\else%
  \ifx\@b\@E (\ref{#1})\else
  \ifx\@b\@F Fact~\ref{#1}\else%
  \ifx\@b\@L Lemma~\ref{#1}\else%
  \ifx\@b\@P Proposition~\ref{#1}\else%
  \ifx\@b\@Q Question~\ref{#1}\else%
  \ifx\@b\@R Remark~\ref{#1}\else%
  \ifx\@b\@S Section~\ref{#1}\else%
  \ifx\@b\@T Theorem~\ref{#1}\else%
  \ifx\@b\@TT Theorem~\ref{#1}\else%
  \ifx\@b\@X Example~\ref{#1}\else%
  \ifx\@b\@s \S\ref{#1}\else
  \ifx\@b\@f Figure~\ref{#1}\else%
  \ifx\@b\@H Conjecture~\ref{#1}\else%
  \ref{#1}\fi\fi\fi\fi\fi\fi\fi\fi\fi\fi\fi\fi\fi\fi\fi}

\def\@rfs#1#2:#3;;{\def\@b{#2}
  \ifx\@b\@C Corollaries~\ref{#1}\else%
  \ifx\@b\@CC Corollaries~\ref{#1}\else%
  \ifx\@b\@F Facts~\ref{#1}\else%
  \ifx\@b\@L Lemmas~\ref{#1}\else%
  \ifx\@b\@P Propositions~\ref{#1}\else%
  \ifx\@b\@Q Questions~\ref{#1}\else%
  \ifx\@b\@R Remarks~\ref{#1}\else%
  \ifx\@b\@S Sections~\ref{#1}\else%
  \ifx\@b\@T Theorems~\ref{#1}\else%
  \ifx\@b\@TT Theorems~\ref{#1}\else%
  \ifx\@b\@X Examples~\ref{#1}\else%
  \ifx\@b\@s \S\ref{#1}\else
  \ifx\@b\@f Figures~\ref{#1}\else%
  \ref{#1}\fi\fi\fi\fi\fi\fi\fi\fi\fi\fi\fi\fi\fi}

\def\@rfF<#1>#2;;{\def\@c{#2}
  \@rfs{#1}#1:;;\ifx\@c\empty\else\@rfL:#2;;\fi}

\def\@rfL:#1<#2>#3;;{\def\@b{#2}\def\@c{#3}
  #1\ifx\@b\empty\else\ref{#2}\ifx\@c\empty\else\@rfL:#3;;\fi\fi}

\catcode`\@=12       

\def\vint_#1{\mathchoice
          {\mathop{\vrule width 6pt height 3 pt depth -2.5pt
                  \kern -8pt \intop}\nolimits_{\kern -4pt#1}}%
          {\mathop{\vrule width 5pt height 3 pt depth -2.6pt
                  \kern -6pt \intop}\nolimits_{#1}}%
          {\mathop{\vrule width 5pt height 3 pt depth -2.6pt
                  \kern -6pt \intop}\nolimits_{#1}}%
          {\mathop{\vrule width 5pt height 3 pt depth -2.6pt
                   \kern -6pt \intop}\nolimits_{#1}}}